\documentclass{elsarticle}
\usepackage{mathtools}
\usepackage{amsmath}
\usepackage{amsbsy}
\usepackage{amsthm}
\usepackage{amssymb}
\usepackage{amsfonts}
\usepackage{bm}
\usepackage{threeparttable}
\usepackage{enumerate}
\usepackage{mathrsfs}
\usepackage[all]{xy}
\usepackage{color}
\usepackage[rgb]{xcolor}
\usepackage{float}
\usepackage{graphicx}
\usepackage{subfig}
\usepackage{bookmark}
\usepackage{multirow}
\usepackage{diagbox}
\usepackage{booktabs}

\numberwithin{equation}{section}
\newtheorem{theorem}{Theorem}[section]
\newtheorem{lemma}[theorem]{Lemma}
\newtheorem{remark}{Remark}[section]
\newdefinition{definition}{Definition}
\newdefinition{example}{Example}[section]

\begin{document}

\begin{frontmatter}
\title{Provably size-guaranteed mesh generation with superconvergence}

\author[nwpu]{Xiangrong Li}
\ead{xiangrong@mail.nwpu.edu.cn}
\author[sdu]{Nan Qi}
\author[nwpu]{Yufeng Nie\corref{cor1}}
\ead{yfnie@nwpu.edu.cn}
\author[nwpu]{Weiwei Zhang}

\cortext[cor1]{Corresponding author}
\address[nwpu]{Research Center for Computational Science, Northwestern Polytechnical University, Xi'an, Shannxi 710129, P.R. China.}
\address[sdu]{Institute of Marine Science \& Technology, Shandong University, Qingdao, Shandong 250100, P.R. China}

\begin{abstract}
The properties and applications of superconvergence on size-guaranteed Delaunay triangulation generated by bubble placement method (BPM), are studied in this paper. First, we derive a mesh condition that the difference between the actual side length and the desired length $h$ is as small as ${\cal O}(h^{1+{\alpha}})$ $({\alpha}>0)$. Second, the superconvergence estimations are analyzed on linear and quadratic finite element for elliptic boundary value problem based on the above mesh condition. In particular, the mesh condition is suitable for many known superconvergence estimations of different equations. Numerical tests are provided to verify the theoretical findings and to exhibit the superconvergence property on BPM-based grids.
\end{abstract}


\end{frontmatter}

\section{Introduction}
Superconvergence of finite element solutions to partial differential equations has been studied intensively for many decades \cite{Douglas1973Superconvergence, Levine1985Superconvergent, Wahlbin1995Superconvergence, Lin2004Natural}. It is shown to be an important tool to develop high-performance finite elements. Various postprocessing techniques exist to recover finite element solutions or their derivatives in order to improve accuracy, such as the popular Zienkiewicz-Zhu (ZZ) method for gradient recovery \cite{zhu1990superconvergence, Zienkiewicz1992The, zienkiewicz1992superconvergent}. 

In early works, there appears a dilemma: the classic superconvergence theory has been usually adopted to specially structured grids, such as the strongly regular grids composed of equilateral triangles \cite{Huang_2008}, but the mesh generation techniques is very hard to satisfy this requirement. Thus there is a serious gap between theory of superconvergence and mesh generation. 

Fortunately, the gap is gradually closing up with the development of superconvergence theory and mesh generation technologies. From one hand, one notable example of the development is the work of Bank and Xu \cite{Bank2003Asymptotically, Bank2004Asymptotically} who studied superconvergence on mildly structured grids where most pairs of elements form an approximate parallelogram. They also proved that linear finite element solution is superclose to its linear interpolant of exact solution. Based on mildly structured grids, Xu and Zhang \cite{Xu2003Analysis} established the superconvergence estimations of three gradient recovery operators containing weighted averaging, local ${L^2}$-projection, and local discrete least-squares fitting. Huang and Xu further investigated the superconvergence properties of quadratic triangular element on mildly structured grids \cite{Huang_2008}. From the other hand, the centroidal Voronoi tessellation (CVT)-based methods have been successfully applied to develop high-quality mesh generation \cite{du1999centroidal}, and the property of superconvergence has been only verified numerically on CVT-grids by Huang \cite{huang2008centroidal}. However, there are still some burning problems. The mesh condition of mildly structured grids is hypothetic in the work of Xu, and currently there is few mesh generation technologies which could theoretically meet the mesh conditions of mildly structured grids though mildly structured grids can be generated numerically by some grids generators. Meanwhile, due to lack of the deduction of mesh condition on CVT-based grids, the superconvergence property on CVT-based grids is just verified by numerical examples without any theoretic results. 

In recent years, the so called bubble placement method has been systematically studied by Nie. \cite{nie2010node, liu2010node, qi2014acceleration}. The advantage of BPM is to generate high-quality grids on many complexly bounded 2D and 3D domains and can be easily used in adaptive finite element method and anisotropic problems \cite{Cai2013Numerical, Zhang2014An, zhang2015adaptive, Zhou2016A, wang2016npbs}. In addition, due to the natural parallelism of BPM, computational efficiency has been improved greatly to solve large-scale problems \cite{Nie2014Parallel}. Yet, superconvergence on BPM-based grids has not been explored. The goal of this paper is to analyze a mesh condition on BPM-based grids, such that superconvergence results can be obtained both theoretically and numerically.

In this paper, we will carefully investigate the superconvergence properties on BPM-based grids. Our work has two main steps. In the first step, a mesh condition where the actual length $l_e$ of any edge $e$ and the desired length $h$ differ only by high quantity of the parameter $h$:
\begin{equation}
|l_e-h|={\cal O}{(h^{1+\alpha})}, \alpha >0
\end{equation}
is derived from an established optimal model of BPM. The second major component of our analysis is two superconvergence results for linear finite elements 
\begin{equation}
{\left\| {{u_h} - {u_I}} \right\|_{1,\Omega }} ={\cal O} (h^{1+\min (\alpha, 1/2 )}),
\end{equation}
and quadratic finite elements
\begin{equation}
{\left\| {{u_h} - {\Pi_Q u}} \right\|_{1,\Omega }} ={\cal O}(h^{2+\min (\alpha, 1/2 )}).
\end{equation}
on BPM-based grids, where $u_I$ and ${\Pi_Q u}$ are the piecewise linear and quadratic interpolant for $u$ respectively. These superconvergence results can be used to derive the posteriori error estimates of the gradient recovery operator for many popular methods, like the ZZ patch recovery and the polynomial preserving recovery \cite{Xu2003Analysis,Zhang2004Polynomial}

The rest of this paper is organized as follows. Section 2 gives the derivation process of mesh conditions and superconvergence results on BPM-based grids. Some numerical experiments of elliptic boundary value problem on different computational domains are given in Section 3 and further discussed in Section 4. Conclusions and future works are given in Section 5.

\begin{figure}
	\centering
	\includegraphics[width=0.4\textwidth]{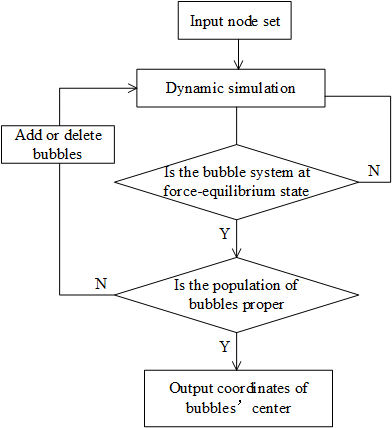}
	\caption{The flowchart of BPM.}\label{step}
\end{figure}

\section{Methodology}
\subsection{Preliminaries}
Bubble placement method was originally inspired by the idea of bubble meshing\cite{shimada1998quadrilateral, Yamakawa2002Quad} and the principle of molecular dynamics. The computational domain is regarded as a force region with viscosity, and bubbles are distributed in the domain. Each bubble is driven by the interaction forces\cite{Shimada1998Automatic} from its adjacent bubbles
\begin{equation}
f\left( w \right) = 
\begin{cases}
{k_0}\left( {1.25{w^3} - 2.375{w^2} + 1.125} \right)  & 0 \le w \le 1.5\\
0  & 1.5 < w
\end{cases} 	
\end{equation}
and its center is taken as one node placed in the domain, where $w=\frac{l_{ij}}{\bar{l_{ij}}}$, $l_{ij}$ is the actual distance between bubble $i$ and bubble $j$, $\bar{l_{ij}}$ is the assigned one given by users. The motion of each bubble satisfies the Newton's second law of motion. BPM can be mainly divided into 3 steps: initialization, dynamic simulation, bubble insertion and deletion operations. And BPM is regarded to be controlled by two nest loops, which is schematically illustrated in Figure \ref{step}. 

The inner loop (dynamic simulation) ensures a good bubble distribution when forces are balanced and the outer loop (insertion and deletion operations) controls the bubble number by adding or deleting bubbles such that adjacent bubbles can be tangent to each other as possible at force-equilibrium state. They both work together to get a closely-packed configuration of bubbles, so that a well-shaped and size-guaranteed Delaunay triangulation can be created by connecting the bubble centers. 


\subsubsection{Inner loop}
In the inner loop, the bubble motion is similar to damped vibrator. In the initial state, there is a potential energy between bubbles, which transforms to kinetic energy during simulation. The motion of bubbles also leads to energy loss as the bubble system has viscous damping force. The potential energy of the bubble system reaches its minimum at force-equilibrium state, and at this moment the resultant force exerting on each interior bubble vanishes.

\begin{figure}
	\centering
	\includegraphics[width=0.4\textwidth]{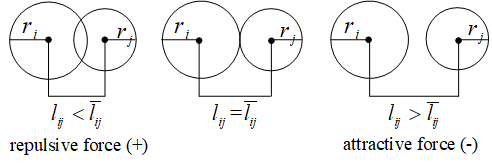}
	\caption{The distance between bubble $i$ and bubble $j$. From left to right, these are overlapping bubbles with repulsive force(the sign of this force is "+" by inter-force formula), tangent bubbles with no force between bubbles and disjoint bubbles with attracting force (the sign of this force is "-") in turn. }\label{distance}
\end{figure}

\begin{figure}
	\centering
	\includegraphics[width=0.4\textwidth]{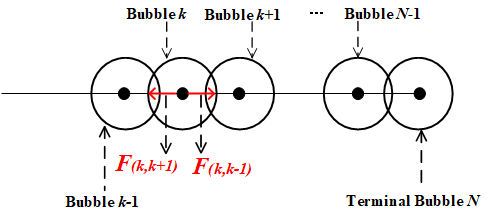}
	\caption{Force-equilibrium state in one-dimensional. For bubble $k$, there is a repulsive force ${F_{( k.k-1)} }$ from bubble $k-1$. When system reach an equilibrium state, resultant external force of the bubble $k$ is zero. Therefore,  ${F_{( k.k+1) }}$ and ${F_{( k.k-1)} }$ must be the same force in magnitude but in opposite direction, whence ${F_{( k.k+1)} }$ should be repulsive. }\label{1D}
\end{figure}

All interior bubbles at force-equilibrium state possess a specific characteristic. For any interior bubble $i$, forces of its adjacent bubbles exerting on it are the same in magnitude and sign. Let us take an 1D case to clarify. For any interior bubble $k$, if it overlaps with its left adjacent bubble $k-1$, there will be repulsive force ${F_{(k.k-1)}}$ between the bubble $k$ and $k-1$. At force-equilibrium state, with the condition of resultant external force of the bubble $k$ vanishing, there must be the same force in magnitude but in opposite direction. So the right adjacent bubble $k+1$ should overlap with the bubble $k$, and ${F_{( k.k+1)}}$ will be the same as ${F_{(k.k-1)}}$ in magnitude but in opposite direction. By inter-force formula (2.1), the sign of inter-force is positive if two adjacent bubbles overlap with each other, otherwise the sign is negative (shown in Figure \ref{distance}). By analogy, each interior bubble is like this until terminal bubble (Terminal bubble is fixed, but there exists interbubble force between terminal bubble and interior bubble.). Figure \ref{1D} visually displays the schematic of above statement. 

Now, let us introduce the definition of bubble fusion degree ${C_{ij}} = {\textstyle{{{{\bar l}_{ij}} - {l_{ij}}} \over {{\bar {l_{ij}}}}}}=1-w$, which characterizes the relative overlapping/disjoint degree of bubble $i$ and bubble $j$. It is easy to derive the following relationship.
\begin{equation} 
\left\{ \begin{array}{l}
{C_{ij}} > 0 \Rightarrow \overline {{l_{ij}}}  > {l_{ij}},\text{Bubble $i$ overlaps with bubble $j$}. \\
{C_{ij}} = 0 \Rightarrow \overline {{l_{ij}}}  = {l_{ij}},\text{Bubble $i$ is tangent to bubble $j$}.\\
{C_{ij}} < 0 \Rightarrow \overline {{l_{ij}}}  < {l_{ij}},\text{Bubble $i$ is disjoint from bubble $j$}.
\end{array} \right.
\end{equation}

Note that if the inter-force between two adjacent bubbles are the same in magnitude and sign, the variable $w$ becomes a constant by the monotonicity of inter-force formula (2.1) in the internal [0, 1.5]. Undoubtedly, now the bubble fusion degree of any two adjacent bubbles is a constant.  

For 2D case, it can be seen as 1D case in any direction, and for all interior bubbles, in whatever direction you choose, the bubble fusion degree of any two adjacent bubbles is a constant like one-dimension. The bubble distribution (or the corresponding node distribution) with this characteristics is called force-equilibrium distribution for convenience.

%
\begin{figure}
	\centering
	\subfloat[$T=0$]{
		\includegraphics[width=0.28\textwidth]{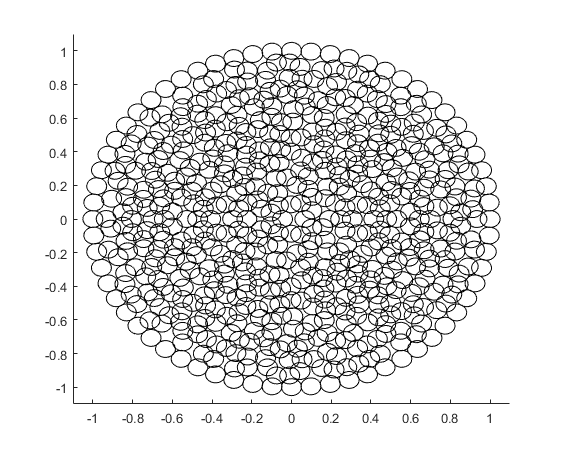}
	}
	\subfloat[$T=100$]{
		\includegraphics[width=0.28\textwidth]{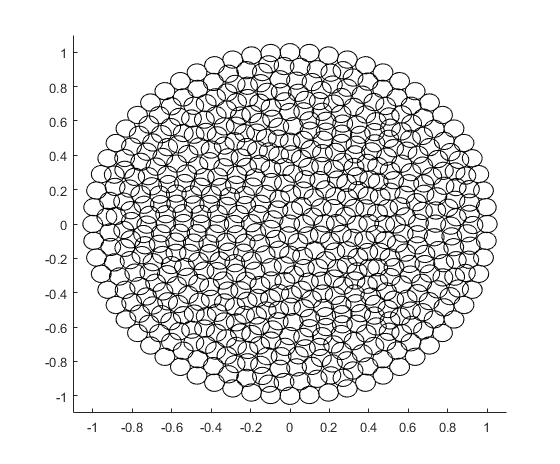}
	}
	\subfloat[$T=200$]{
		\includegraphics[width=0.28\textwidth]{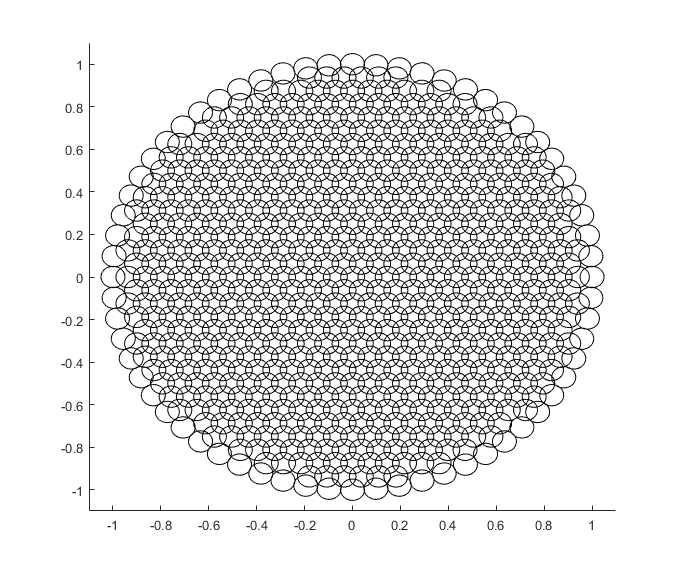}
	}
	
    \subfloat[$T=0$]{
    	\includegraphics[width=0.28\textwidth]{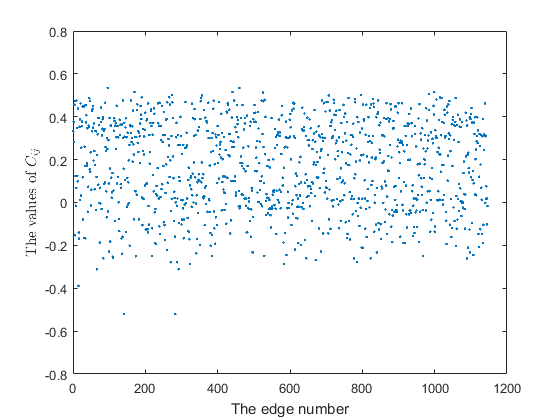}
    }
    \subfloat[$T=100$]{
    	\includegraphics[width=0.28\textwidth]{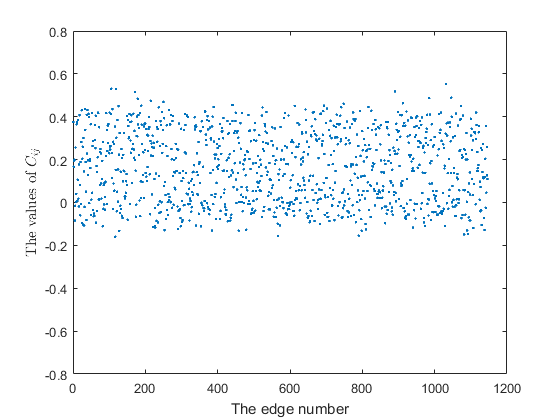}
    }
    \subfloat[$T=200$]{
    	\includegraphics[width=0.28\textwidth]{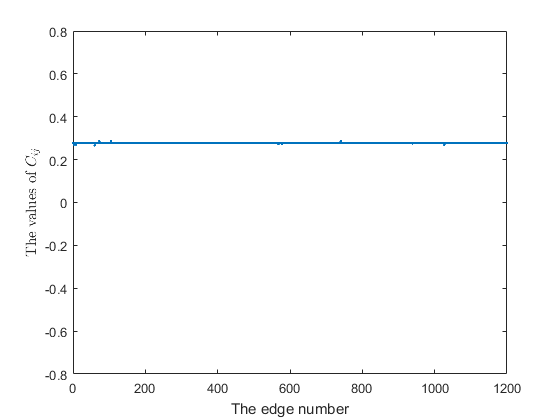}
    }
	\caption{The bubble distributions and their corresponding $C_{ij}$ values at different time steps(now the total number of bubbles $N=606$). }\label{simulation_1}
\end{figure}


Choosing the assigned size function as $d\left( x,y \right)=h=0.1$ to execute BPM algorithm on an unit circle region. As shown in Figure \ref{simulation_1}, the initial bubble distribution is chaotic, and then gradually tends to the force-equilibrium distribution with time step T going on. Meanwhile, the corresponding $C_{ij}$ values tend towards a constant 0.28, which validates our inference.

In a word, bubble system will be at the force-equilibrium state by inner loop. And now the bubble fusion degree of any two adjacent bubbles is also a constant.

\subsubsection{Outer loop}
Let 
\begin{equation}
{\epsilon}^N=\mathop {\max }\limits_{ {i,j} \in {\Gamma}_N}|C_{ij}|,
\end{equation}
where $N$ is the total number of bubbles in the current loop, and $\Gamma_N$ denotes the bubble set at force-equilibrium state with the number $N$. Outer loop aims to control bubbles' number by the overlapping ratio\cite{liu2010node}. More specifically, delete the bubble whose overlapping ratio is larger, or add some new bubbles near the bubble whose overlapping ratio is smaller. If ${\epsilon}^N$ no longer reduces, iterative processes terminate.

\begin{figure}
	\centering
	\subfloat[$N=606$]{
		\includegraphics[width=0.3\textwidth]{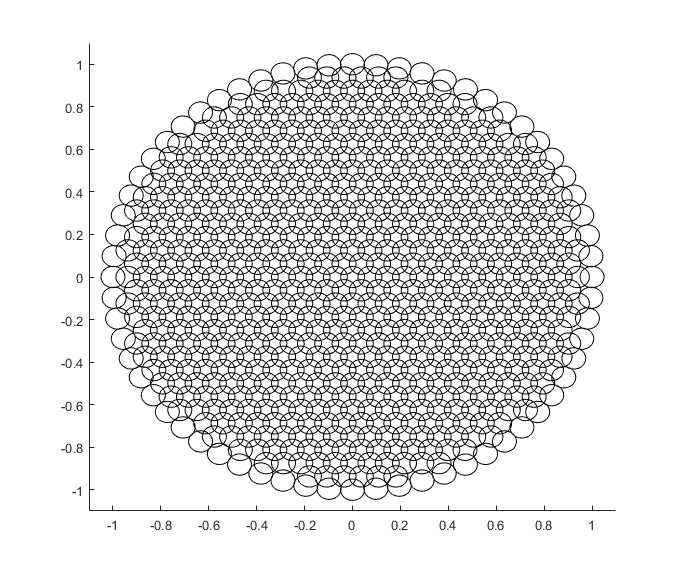}
	}
	\subfloat[$N=488$]{
		\includegraphics[width=0.3\textwidth]{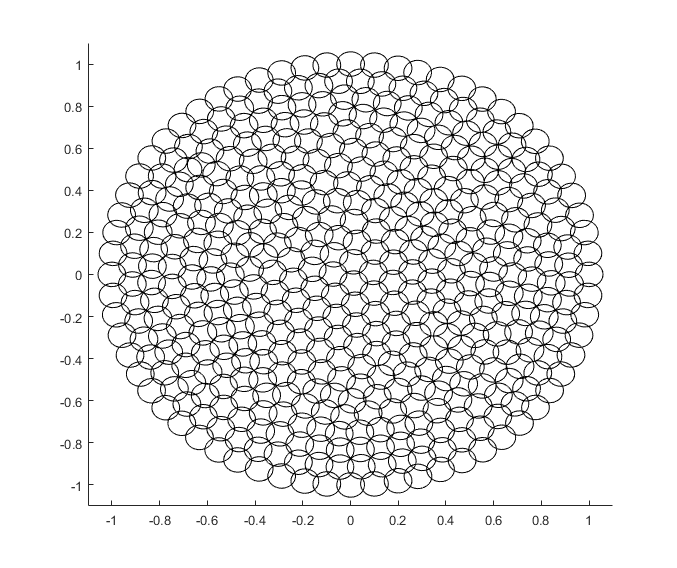}
	}
	\subfloat[$N=472$]{
		\includegraphics[width=0.3\textwidth]{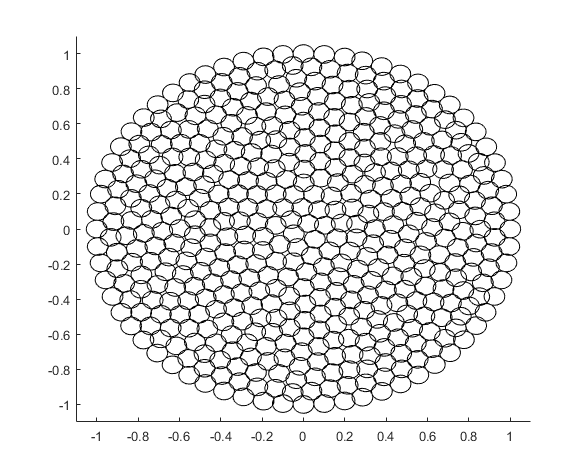}
	}	

	\subfloat[${\epsilon}^N=0.317$]{
		\includegraphics[width=0.3\textwidth]{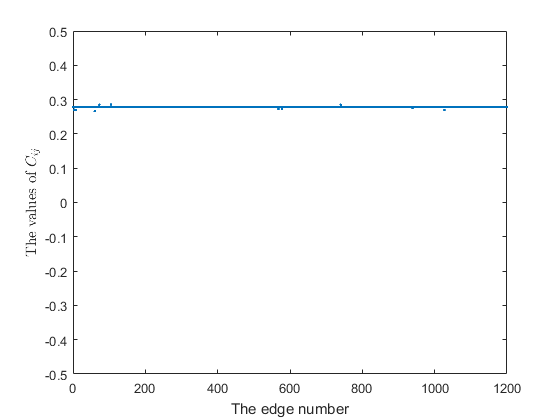}
	}
	\subfloat[${\epsilon}^N=0.186$]{
		\includegraphics[width=0.3\textwidth]{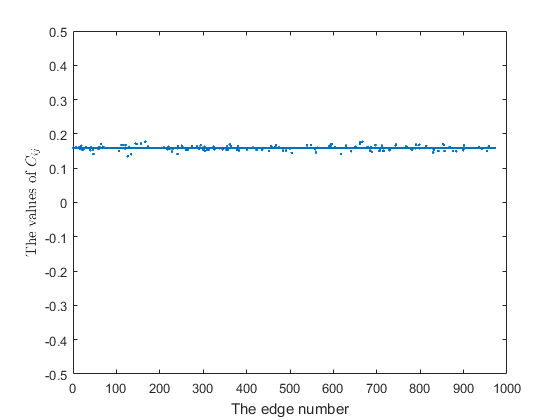}
	}
	\subfloat[${\epsilon}^N=0.112$]{
		\includegraphics[width=0.3\textwidth]{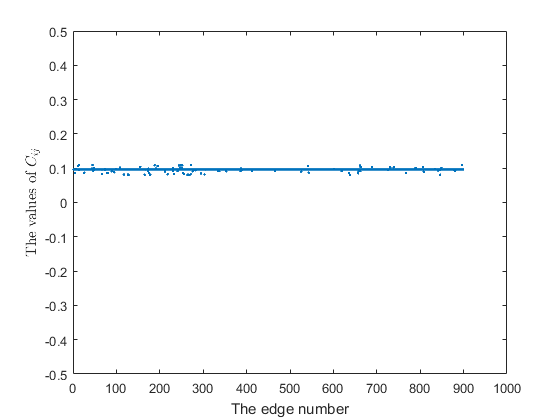}
	}
	\caption{ The bubble distribution their corresponding $C_{ij}$ values after deletion operation. }\label{simulation_2}
\end{figure}

For the same circle region in 2.1.1, the Figure \ref{simulation_2} shows the bubble distribution after several rounds of adding or deleting bubbles. It can be seen that bubbles gradually tend to be tangent without a wide range of overlapping, meaning that the actual distance of two adjacent bubbles is extremely close to their assigned size function. From the changes in $C_{ij}$ , we can clearly see that values of $C_{ij}$ and ${\epsilon}^N$ both decrease, implying that operations of adding or deleting bubbles are effective. 

In summary, BPM can be described mathematically: find a proper number of bubbles ${\bar N}$, such that
\begin{equation}
{\Gamma _{\bar N}} = \left\{ {{\Gamma _N}:\mathop {\min }\limits_N \left\{ {\mathop {\max }\limits_{i,j \in {\Gamma _N}} \left| {{C_{ij}}} \right|} \right\}} \right\},
\end{equation}
where $N$ is the total number of bubbles, $\Gamma_N$ denotes the bubble set at force-equilibrium state with the number $N$, and $\Gamma_{\bar N}$ is the final output.

\begin{remark}
The properties of inner loop and outer loop are also suitable for non-uniform case. For the size function
\begin{equation}
d\left( {x,y} \right) = 
\begin{cases}
0.1 & \sqrt {{x^2} + {y^2}}  < 2,\\
0.2 \times \left| {\sqrt {{x^2} + {y^2}}  - 2} \right| + 0.1 & \sqrt {{x^2} + {y^2}}  \ge 2.
\end{cases}
\end{equation}
We execute the BPM algorithm on a square region $[-3, 3] \times [-3, 3]$, and some numerical evidences are given in Figure \ref{simulation_11}. 
\end{remark}

\begin{figure}[h]
	\centering
	\subfloat[$Nt=200$, $N=953$]{
		\includegraphics[width=0.35\textwidth]{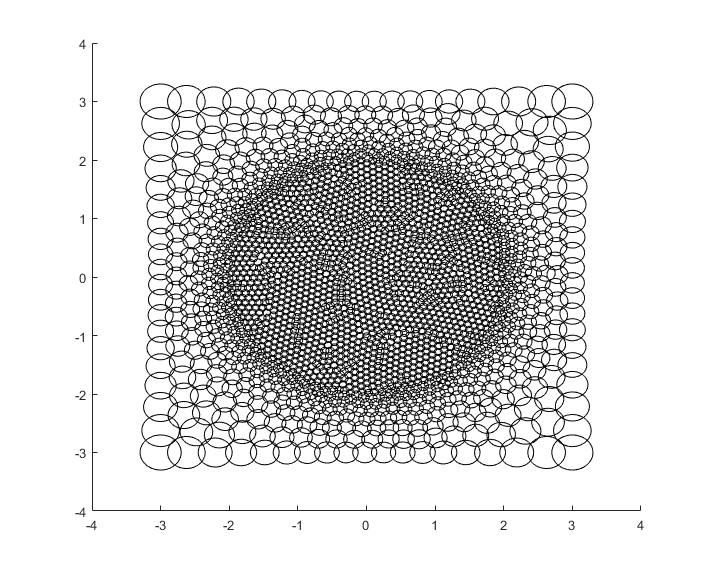}
	}
	\subfloat[$Nt=200$, $N=862$]{
		\includegraphics[width=0.35\textwidth]{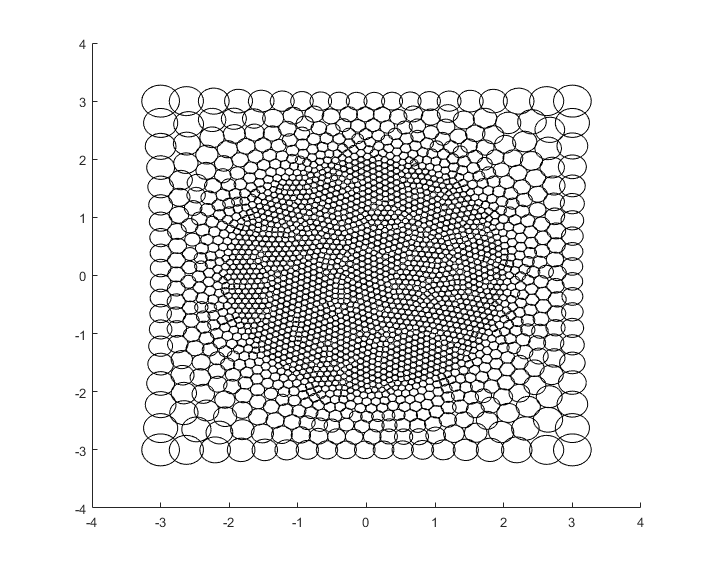}
	}

	\subfloat[$Nt=200$, $N=953$]{
		\includegraphics[width=0.32\textwidth]{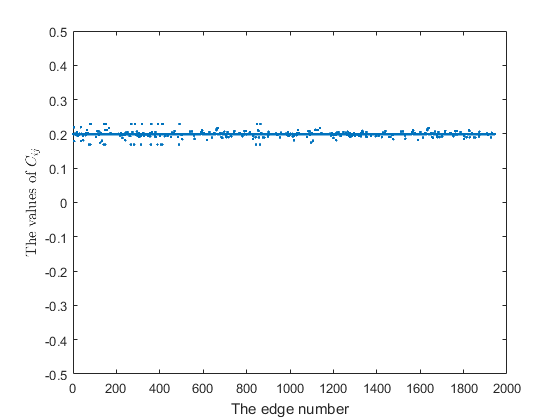}
	}
	\subfloat[$Nt=200$, $N=862$]{
		\includegraphics[width=0.32\textwidth]{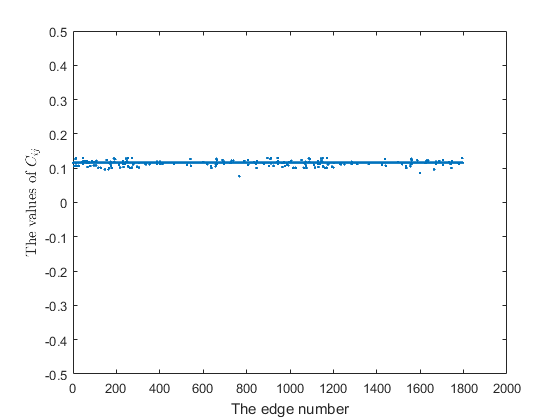}
	}
	\caption{The bubble distributions and their corresponding $C_{ij}$ values. }\label{simulation_11}
\end{figure}

\subsection{The mesh condition of BPM-based grids}

We aware that $\epsilon ^{\bar N}$ is a very important value throughout the simulation. In fact, $\epsilon ^{\bar N}$ is equivalent to the relative errors of all side lengths, so it is very useful to study the mesh condition of BPM-based grids. Actually, $\epsilon ^{\bar N}$ is related to the computational domain and given size function. Next,  $\epsilon ^{\bar N}$ will be discussed separately by different circumstances. For the sake of clearness in the description, we mainly consider uniform distribution (the size function is a constant $h$), so 
\begin{equation*}
{\epsilon}^N=\mathop {\max }\limits_{ {i,j} \in {\Gamma}_N}\left| \dfrac{h-{l_{ij}}}{h}\right|.
\end{equation*}

We firstly define 'ideal subdivision' if the prescribed region can be exactly covered by $\bar N$ equilateral triangles with side size $h$. For instance, an equilateral triangle region with side length $1$ can be divided into 25 equilateral triangles with length $0.2$, so that $\epsilon ^{\bar N}=0$. 

However, the case encountered more frequently is not an 'ideal subdivision'. Such as a square with side as $1$, however the desired size is needed to be $h=0.3$. So we have to look for a subdivision with $\bar N$ elements such that $\epsilon ^{\bar N}$ is optimal. Though we don't know the exact value at first, it is easy to estimate the rough range.

For simplicity, let's start from an 1D case. A domain with length $L$ is required to be uniformly divided into several elements with size $h$. Let $N_e = \lfloor {\frac{L}{h}} \rfloor $ be the number of elements, and the remaining part after uniform subdivision $l = L - {N_e} \cdot h$, then $l \in \left[ {0,h} \right)$ (if $l = 0$, that is 'ideal subdivision'). The element $\delta$ with length $l$ has an error $e_{\delta}$, although other elements are ideal. So the mesh error $e=\max \left\lbrace e_{\delta},0\right\rbrace = \left| {h - l} \right| = \left| {h - \left( {L - {N_e} \cdot h} \right)} \right| = \left| {\left(  {{N_e} + 1} \right) \cdot h - L} \right| = O\left( h \right)$. 

\begin{figure}
	\centering
	\subfloat[before error averagely distributed]{
		\includegraphics[width=0.35\textwidth]{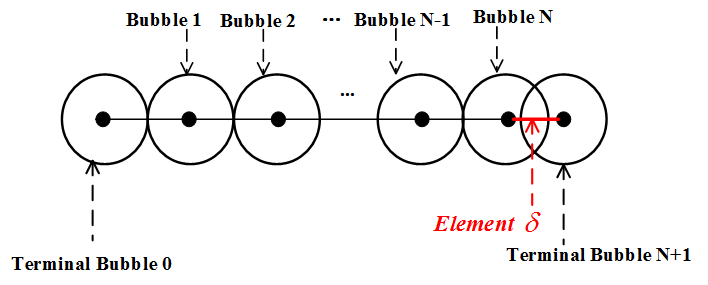}
	}
	\subfloat[after error averagely distributed]{
		\includegraphics[width=0.35\textwidth]{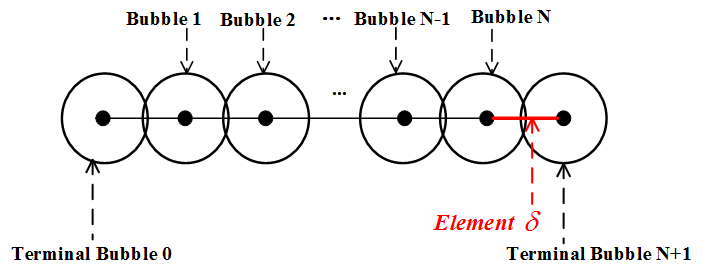}
	}
	\caption{ (a) shows the bubble distribution without force-equilibrium characteristics and (b) displays the bubble distribution at force-equilibrium state. Meanwhile, the distance of two adjacent bubbles is equivalent to the length of corresponding element.}\label{1D_errordistribution}
\end{figure}

For BPM-based grids, the bubble fusion degree of any two adjacent bubbles is a constant at force-equilibrium state, so the error of each element $|l_{ij}-h|$ is a constant, implying that the mesh error gets averaged overall elements, which is also elaborated in Figure \ref{1D_errordistribution}. Thus
\begin{equation*}
e = \left| {h - l_{ij}} \right| =  \left| {\frac{{\left( {{N_e} + 1} \right) \cdot h - L}}{{{N_e} + 1}}} \right| \le \frac{h}{{{N_e} + 1}} \le \frac{h}{{N_e}} = \frac{h}{{\left\lfloor {\frac{L}{h}} \right\rfloor }} = O\left( {{h^2}} \right),
\end{equation*}
which presents a higher accurancy.

As to 2D domain. A plane domain with area $S$ is required to be uniformly divided into several equilateral triangles whose side length is $h$. We know that the area of a equilateral triangle ${s_t}$ is $ \frac{{\sqrt 3 }}{4}{h^2} = O\left( {{h^2}} \right)$. However, the value of ${N_e} = \lfloor {\frac{S}{{{s_t}}}} \rfloor $ is not a good estimation since this way ignore the influence from the remaining part near boundaries of the prescribed region. For instance, dividing a unit circle into equilateral triangles with side length 0.3 start from its center, near the boundary there always remains a ring-like area that can not be exactly covered by equilateral triangles with length 0.3. So $\lfloor {\frac{S}{{{s_t}}}} \rfloor $ overstimates, and it shouldd be replaced by ${N_e} = \lfloor {\frac{S}{{{s_t}}}} \rfloor  - n$, $n\in {Z^ + }$ with $n \ll {N_e}$. Let $s = S - {N_e} \cdot {s_t}$ be the area of the remaining part, then $0 \le s < \left( {n + 1} \right){s_t}$. Similar to the 1D example, we have:
\begin{equation*}
\mathop {\lim }\limits_{h \to 0} \frac{{{e}}}{{{h^2}}} = \mathop {\lim }\limits_{h \to 0} \frac{{\left| {{\textstyle{{S-{N_e} \cdot s_t} \over {{N_e}}}}} \right|}}{{{h^2}}} \le \mathop {\lim }\limits_{h \to 0} \frac{{{\textstyle{{\left( {n + 1} \right){s_t}} \over {{N_e}}}}}}{{{h^2}}} \le \mathop {\lim }\limits_{h \to 0} \frac{{{\raise0.5ex\hbox{$\scriptstyle {\left( {n + 1} \right){s_t}}$}
			\kern-0.1em/\kern-0.15em
			\lower0.25ex\hbox{$\scriptstyle {\left\lfloor {\frac{S}{{{s_t}}}} \right\rfloor }$}}}}{{{h^2}}} = 0.
\end{equation*}
Consequently, the error of edge ${e_h} = \left| {h - l_e} \right|={\cal O}(h^{1+\alpha}), (\alpha>0)$, where $l_e$ is the actual length of each edge in BPM-based grids. Above results are true for all element, so for the current node number $N$ we have ${\epsilon}^N={\cal O}(h^{\alpha})$. At the moment $\epsilon ^N $ is not necessarily equal to $\epsilon ^{\bar N}$ which we are looking for in our model, but the following inequality is certainly true:
\begin{equation*}
\epsilon ^{\bar N} \le {\epsilon ^N}.
\end{equation*}

With these, the well-spaced bubbles, the bubble fusion degree of any two adjacent bubbles $|C_{ij}|={\cal O}(h^{\alpha})$, can be generated by BPM, so that size-guaranteed grids which the actual length of each edge and the given size $h$ differ only by high quantity of the parameter $h$ 
\begin{equation}
|h-l_e|={\cal O}(h^{1+\alpha})
\end{equation}
could be created by connecting the centers of bubbles. Naturally, grids are in good shape. 

\begin{remark}
In this paper, $l_{ij}$ means the distance between bubble $i$ and $j$, and $l_e$ denote the length of edge $e$. However, they are essentially equivalent but in a different form, which are clearly visible in Figure \ref{Notation}(a).
\end{remark}

\begin{remark}
For non-uniform case, it is no longer error distributed averagely, but distributed with different weights:
\begin{equation*}
{\lambda _e} = \frac{{{{\bar l}_e}}}{{\sum\limits_{k = 1}^{N_e} {{{\bar l}_k}} }}
\end{equation*}
where $\lambda_e$ is the weight for edge $e$, ${\bar {l_e}}$ is the desired length of edge $e$ (computed by the size function), ${N_e}$ means the total number of elements. In addition, we should assume that for any elements, the desired length of their three sides satisfy ${\bar {l_{e + 1}}} \simeq {\bar {l_e}} \simeq {\bar {l_{e - 1}}}$. This assumption is in accordance with the case of mesh refinement during adaptive iterations and many non-uniform triangulations. The rest will be treated as uniform case, and it will draw a conclusion that for any edge $e$, $|{\bar {l_e}}-{l_e}|={\cal O}({\bar {l_e}}^{1+{\alpha}}) ({\alpha}>0)$.
\end{remark}


\subsection{Superconvergence on BPM-based grids}
Let us consider an interior edge shared by two elements $\tau $ and$\tau '$, shown in Figure \ref{Notation}(b). $l_e$ denote the length of edge $e$. For the element $\tau$, let ${l_{e-1}}$ and ${l_{e+1}}$ be length of two other edges. With respect to $\tau '$, ${l_{e' - 1}}$ and ${l_{e' + 1}}$ are also length of two other edges.

\begin{figure}
	\centering
	\subfloat[]{
		\includegraphics[width=0.4\textwidth]{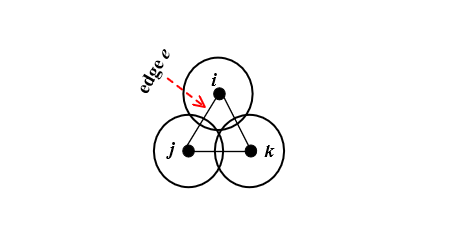}
	}
	\subfloat[]{
		\includegraphics[width=0.3\textwidth]{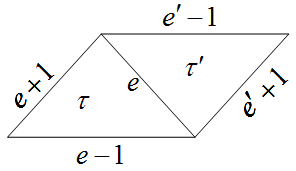}
	}
	\caption{ Notations }\label{Notation}
\end{figure}

Properties. Suppose we have a triangulation ${{\cal T}_h}$ is generated by BPM, and from Eq.(2.6),
\begin{enumerate}
	\item  For element $\tau$, using the triangle inequality, we have
	\begin{equation*}
	\left| {{l_{e + 1}} - {l_{e - 1}}} \right| = {\cal O}({h^{1 + \alpha }})
	\end{equation*}
	
	\item For two elements $\tau$ and $\tau '$, using the triangle inequality, we have
	\begin{equation*}
	\left| {{l_{e + 1}} - {l_{e' + 1}}} \right| = {\cal O}({h^{1 + \alpha }})
	\end{equation*}
\end{enumerate}

These properties are analogous to the definition of mildly structured grids \cite{Xu2003Analysis}, which pursues that two adjacent triangles (sharing a common edge) form an ${\cal O}(h^{1+\alpha})$ $({\alpha}>0)$ approximate parallelogram, i.e., the lengths of any two opposite edge differ only by ${\cal O}(h^{1+\alpha})$. Consequently, it is very easy to apply our BPM-based grids to all superconvergence estimations of mildly structured grids without any changes. 

Let $\Omega  \subset {{R}^2}$ be a bounded polygon with boundary $\partial \Omega $. Consider problem: Find $u \in V$ such that
\begin{equation}
a\left( {u,v} \right) =\int_\Omega  {\nabla u\nabla v{\kern 1pt} dx} = \left( {f,v} \right),\forall v \in V,
\end{equation}
where  $\left( { \cdot , \cdot } \right)$ denotes inner product in the space ${L^2}\left( \Omega  \right)$, and $V \subset {H^1}\left( \Omega  \right)$, if boundary conditions is different, $V$ is a little different. It is known that $a\left( { \cdot , \cdot } \right)$ is a bilinear form which satisfies the following conditions:
\begin{enumerate}
	\item (Continuity). There exists $C \ge 0$ such that
	\begin{equation*}
	\left| {a\left( {u,v} \right)} \right| \le C{\left\| u \right\|_{1,\Omega }}{\left\| v \right\|_{1,\Omega }},
	\end{equation*}
	for all $u,v \in V$.
	\item (Coerciveness). There exists $M > 0$ such that
	\begin{equation*}
	a\left( {v,v} \right) \ge M\left\| v \right\|_{_{1,\Omega }}^2,\forall v \in V.
	\end{equation*}
\end{enumerate}

Let ${V_h}^k =\{v_h:v_h \in {H^1}(\Omega),v_h|_{\tau} \in P_k(\tau)\}$, $k=1,2$, be the conforming finite element space associated with triangulation ${{\cal T}_h}$. Here $P_k$ denotes the set of polynomials with degree $\leq k$. The finite element solution ${u_h} \in {V_h}^k$ satisfies  
\begin{equation}
a\left( {{u_h},v} \right) = \left( {f,v} \right),\forall v \in {V_h}^k.
\end{equation}

The following two lemmas about some superconvergence results on linear and quadratic elements for Poisson problems, are a simple modification of [11, Lemma 2.1] and [8, Theorem 4.4].

\begin{lemma}
For triangulation ${{\cal T}_h}$ generated by bubble-type mesh generation, for any ${v_h} \in {V_h}^k$
\begin{equation}
\left| {\int_\Omega  {\nabla ( {u - {u_I}} ) \nabla {v_h}} } \right| \lesssim {h^{1+\min (\alpha, 1/2 )}}{\| {{v_h}} \|_{1,\Omega }}.
\end{equation}
where ${u_I}$ is the linear interpolation of $u$ and $k=1$.
\end{lemma}

\begin{lemma}
For triangulation ${{\cal T}_h}$ generated by bubble-type mesh generation,  for any ${v_h} \in {V_h}^k$
\begin{equation}
\left| {\int_\Omega  {\nabla ( {u - {\Pi_Q u}} ) \nabla {v_h}} } \right| \lesssim {h^{2+\min (\alpha, 1/2 )}}{\| {{v_h}} \|_{1,\Omega }}.
\end{equation}
where ${\prod_Q u}$ is the quadratic interpolation of $u$ and $k=2$.
\end{lemma}

\begin{remark}
The arguments for these lemmas are the same as [11, Lemma 2.1] and [8, Theorem 4.4], and some adjustments we should make are trivial. Here it's no need to say more, for details see in reference \cite{Xu2003Analysis,Huang_2008}.
\end{remark}	

\begin{theorem}
Assume that the solution of (2.7) satisfies $u \in {H^3}\left( \Omega  \right) \cap W_\infty ^2\left( \Omega  \right)$, and ${u_h}$ is the solution of (2.8). Let ${u_I} \in {V_h}^1$ and ${\Pi_Q u} \in {V_h}^2$ be the linear  and quadratic interpolation of $u$, respectively. For triangulation ${{\cal T}_h}$ derived from BPM-based grids, we have
\begin{equation}
{\left\| {{u_h} - {u_I}} \right\|_{1,\Omega }} ={\cal O} (h^{1+\min (\alpha, 1/2 )}),
\end{equation}
and
\begin{equation}
{\left\| {{u_h} - {\Pi_Q u}} \right\|_{1,\Omega }} ={\cal O}(h^{2+\min (\alpha,  1/2 )}).
\end{equation}
\end{theorem}
\begin{proof}
	Taking $v_h={u_h}-{u_I}$ in Lemma 2.1, we have
	\begin{equation*}
	\begin{split}
	{\left\| {{u_h} - {u_I}} \right\|_{1,\Omega }}^2
	&=a(u_h-u_I, u_h-u_I)
	=a(u-u_I, u_h-u_I)\\
	&=\left| {\int_\Omega  {\nabla \left( {u - {u_I}} \right) \nabla {u_h-u_I}} }  \right|\\
	&\lesssim {h^{1+\min (\alpha, 1/2 )}}{\| {{u_h-u_I}} \|_{1,\Omega }}
	\end{split}	
	\end{equation*}
	The proof is completed by canceling $\|u_h-u_I\|_{1,\Omega }$ on both sides of the inequality. Similar, taking $v_h={u_h}-{\Pi_Q u}$ in Lemma 2.2, (2.10) can be easily obtained. 
\end{proof}

\section{Numerical examples}
In this section, we will report some numerical examples to support theoretical estimations and verify the superconvergence property of solving Poisson equation on BPM-based grids. The examples considered vary from 'ideal subdivision', such as a unit equilateral triangle region, to 'non-ideal subdivision'. 

In order to measure mesh shape quality simply and clearly, we define the mesh shape quality measure as the ratio between the radius of the largest inscribed circle (times two) and the smallest circumscribed circle \cite{Persson2004A}, which is very similar to the concept of `radius ratio':
\begin{equation*}
q{(a,b,c)} = \frac{{2{r_{in}}}}{{{r_{out}}}} = \frac{{{(b + c - a)(c + a - b)(a + b - c)}}}{{{abc}}}
\end{equation*}
where $a$, $b$, $c$ are the side lengths. An equilateral triangle has $q=1$. Define the average mesh quality over placement area:
\begin{equation*}
{Q_{avg}} = \frac{1}{M}\sum\limits_{m = 1}^M {{q_m}} 
\end{equation*}
where $M$ represents the number of elements, and ${q_m}$ is the mesh shape quality of the $mth$ element. In addition, the closer that ${Q_{avg}}$ value is to 1, the more regular the grid is. 

\subsection{Example 3.1: An unit equilateral triangle region}
\begin{figure}[h]
	\centering
	\subfloat[$h=0.2$]{
		\includegraphics[width=0.3\textwidth]{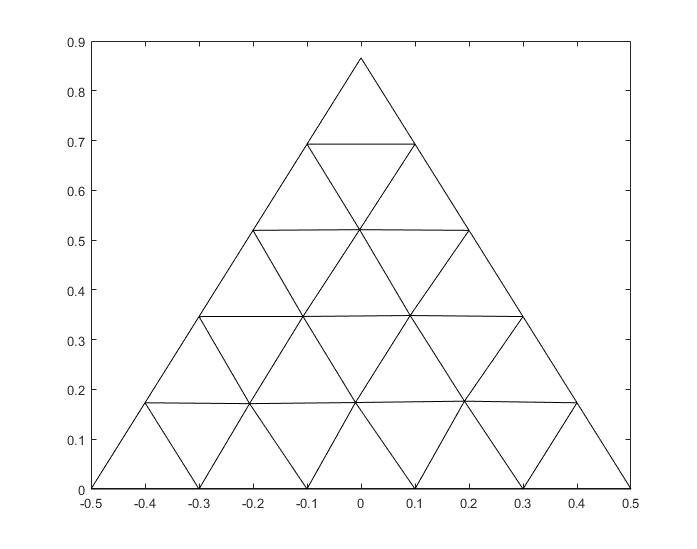}
	}
	\subfloat[$h=0.1$]{
		\includegraphics[width=0.3\textwidth]{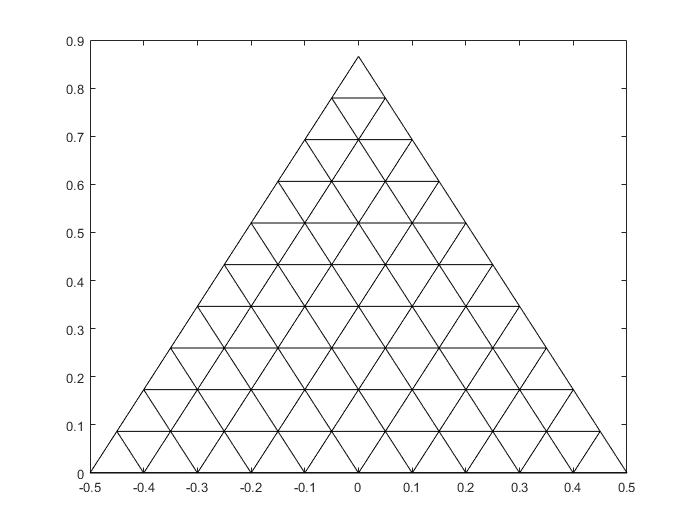}
	}
	
	\subfloat[$h=0.05$]{
		\includegraphics[width=0.3\textwidth]{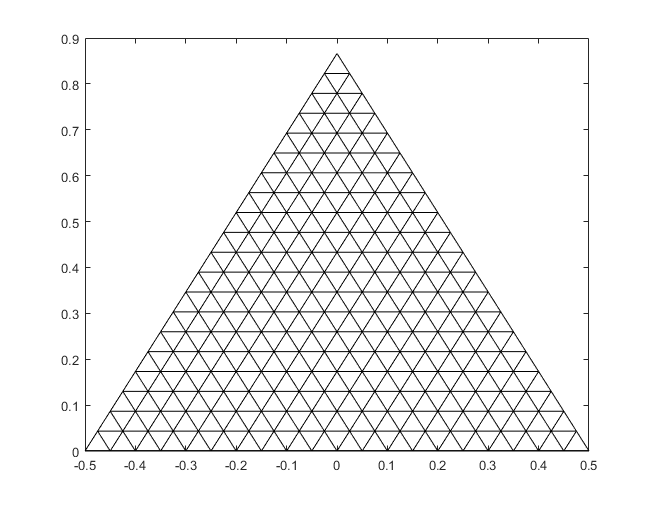}
	}
	\subfloat[$h=0.025$]{
		\includegraphics[width=0.3\textwidth]{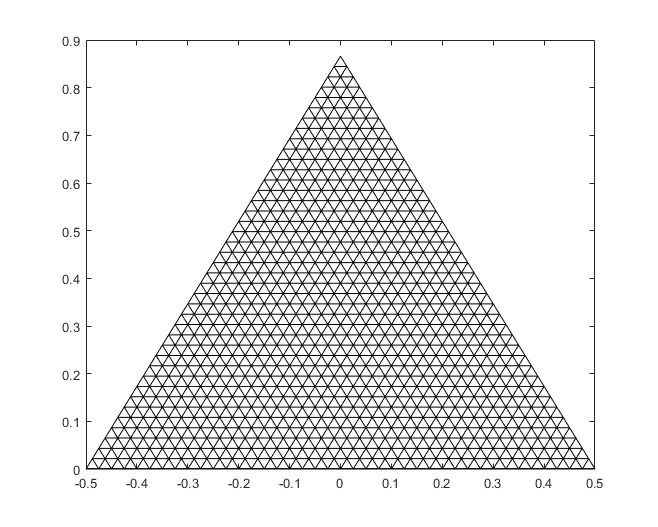}
	}
	\caption{ BPM-based grids on square region in different sizes.}\label{grids_t}
\end{figure}
	
This is an equilateral triangle region with side length 1 and we solve Poisson equation on it with Dirichlet boundary conditions. The right-hand side $f$ and the boundary conditions are chosen such that the exact solution is $u = {\cos 2 \pi x \sin 2 \pi y}$. Taking initial size $h=0.2$, when $h$ reduces by half in turn, BPM-based grid configurations with the first four sizes are selectively shown in Figure \ref{grids_t}. Obviously, the near-perfect grids can be generated by our algorithm for 'ideal subdivision' case. 

\begin{table}
	\caption{ Superconvergence results for equilateral triangle region}
	\centering
	\small
	\begin{tabular}{ccccccc}
		\toprule
		$h$  & ${\left\| {{u_h} - {u_I}} \right\|_{1,\Omega }}$ & order (k=1) & ${\left\|  {{u_h} - {\Pi_Q u}} \right\|_{1,\Omega }}$ & order (k=2) & ${Q_{avg}}$\\
		\midrule
		0.2    &  1.81E-01 &       &  8.93E-03 &      & 0.9998 \\
		0.1    &  4.30E-02 &  2.08 &  1.14E-03 &  2.97& 1.0000 \\
		0.05   &  1.05E-02 &  2.05 &  1.41E-04 &  3.02& 1.0000 \\
		0.025  &  2.63E-03 &  2.01 &  1.78E-05 &  2.96& 1.0000 \\
		0.0125 &  6.55E-04 &  2.01 &  2.30E-056 &  2.95& 1.0000 \\
		\bottomrule
	\end{tabular}\label{tab2}
\end{table}

Some numerical results are given in Table \ref{tab2}, where ${u_I}$ and ${\Pi_Q u}$ are the linear and quadratic interpolant of $u$, respectively. From Table \ref{tab2}, we see quite clearly the superconvergence of  ${\left\| {\nabla \left( {{u_I} - {u_h}} \right)} \right\|_{0,\Omega }}$ and ${\left\| { {{u_h} - {\Pi_Q u}}} \right\|_{1,\Omega }}$, with order close to 2 and 3, respectively. 
Note that when h is smaller than 0.2, the corresponding $Q_{avg}$ can achieve to 1. It indicates that the superconvergence property is inseparable from the regularity of grids.

\subsection{Example 3.2: A unit circle region centered at origin}
	\begin{figure}
		\centering
		\subfloat[$h=0.2$]{
			\includegraphics[width=0.3\textwidth]{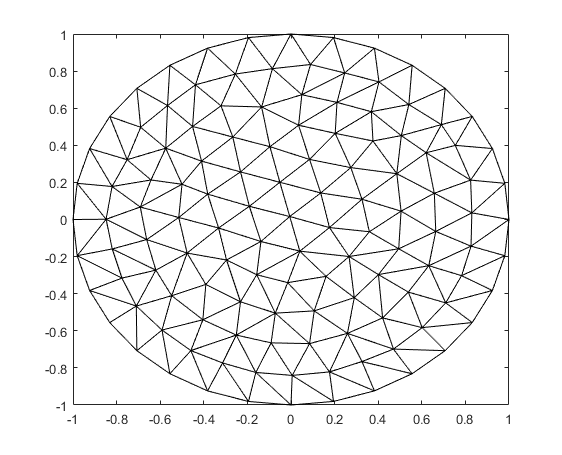}
		}
		\subfloat[$h=0.1$]{
			\includegraphics[width=0.3\textwidth]{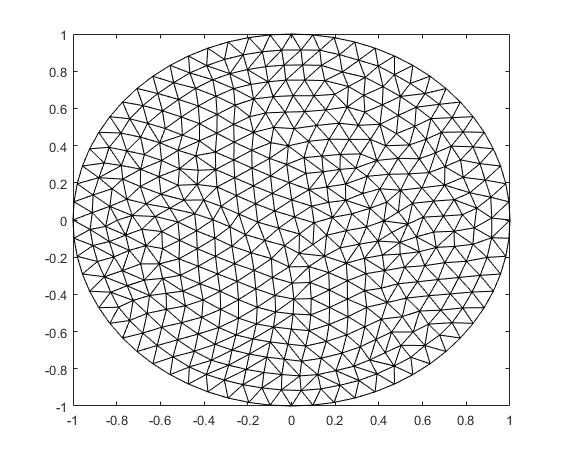}
		}
		
		\subfloat[$h=0.05$]{
			\includegraphics[width=0.3\textwidth]{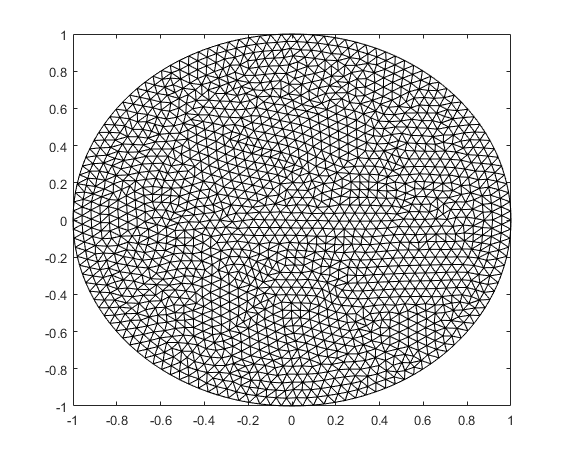}
		}
		\subfloat[$h=0.025$]{
			\includegraphics[width=0.3\textwidth]{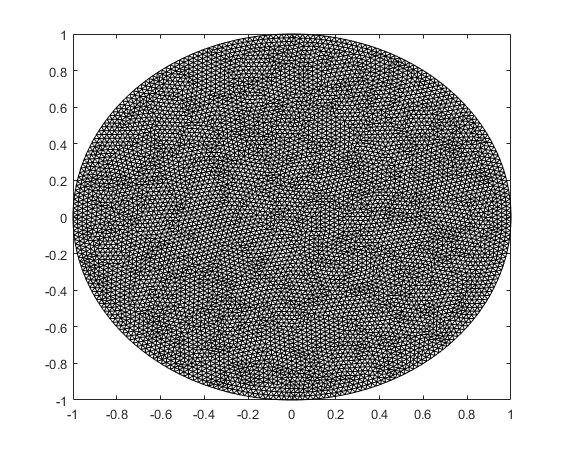}
		}
		\caption{BPM-based grids on circle region in different sizes.}\label{grids_c}
	\end{figure}
\begin{table}[h]
	\caption{Results for circle region}
	\centering
	\small
	\begin{tabular}{ccccccc}
		\toprule
		$h$  & ${\left\| {{u_h} - {u_I}} \right\|_{1,\Omega }}$ & order (k=1) & ${\left\|  {{u_h} - {\Pi_Q u}} \right\|_{1,\Omega }}$ & order (k=2) & ${Q_{avg}}$\\
		\midrule
		0.2    &  1.09E-01 &       &  1.93E-02 &      & 0.9510 \\
		0.1    &  3.93E-02 &  1.47 &  3.53E-03 &  2.45& 0.9635 \\
		0.05   &  1.36E-02 &  1.53 &  6.16E-04 &  2.52& 0.9732 \\
		0.025  &  4.82E-03 &  1.50 &  1.09E-04 &  2.50& 0.9702 \\
		0.0125 &  1.69E-03 &  1.51 &  1.96E-05 &  2.47& 0.9753 \\
		\bottomrule
	\end{tabular}\label{tab_c}
\end{table}

For a unit circle region centered at origin, the size values are taken by $0.2$, $0.1$, $0.05$, $0.025$, $0.0125$ respectively. Figure \ref{grids_c} shows that BPM-based grids are in good shape generally. Choosing the exact solution $u=sinxsiny$, some calculated results are given in Table \ref{tab_c}. Obviously, there is superconvergence phenomenon on BPM-based grids and the results clearly indicate that ${\left\| { {{u_h} - {u_I}}} \right\|_{1,\Omega }}$ and ${\left\| { {{u_h} - {\Pi_Q u}}} \right\|_{1,\Omega }}$ are close to ${\cal O}({{h^{1.50}}})$ and ${\cal O}({{h^{2.50}}})$ roughly.
Although the convergence order is lower than example 3.1, these results are still consistent with theoretic estimations (2.11) and (2.12). 

For all edges $\cal E$, denote $h_{err} = {\textstyle{{\sum {\left| {{l_e} - h} \right|} } \over {\# {\cal E}}}}$, i.e., the mean value of all edges' error. Figure \ref{error_cp}(a) shows the relationship between $h_{err}$, ${\left\| { {{u_h} - {u_I}}} \right\|_{1,\Omega }}$ and ${\left\| { {{u_h} - {\Pi_Q u}}} \right\|_{1,\Omega }}$. Graphically, ${\left\| {\nabla \left( {{u_I} - {u_h}} \right)} \right\|_{0,\Omega }}$ and $h_{err}$ have similar tendency and ${\left\| { {{u_h} - {\Pi_Q u}}} \right\|_{1,\Omega }}$ is more abrupt than ${\left\| { {{u_h} - {u_I}}} \right\|_{1,\Omega }}$, which illustrate the validity of superconvergence estimation in Theorem 2.1.

\begin{figure}
	\centering
	\subfloat[A unit circle region]{
		\includegraphics[width=0.4\textwidth]{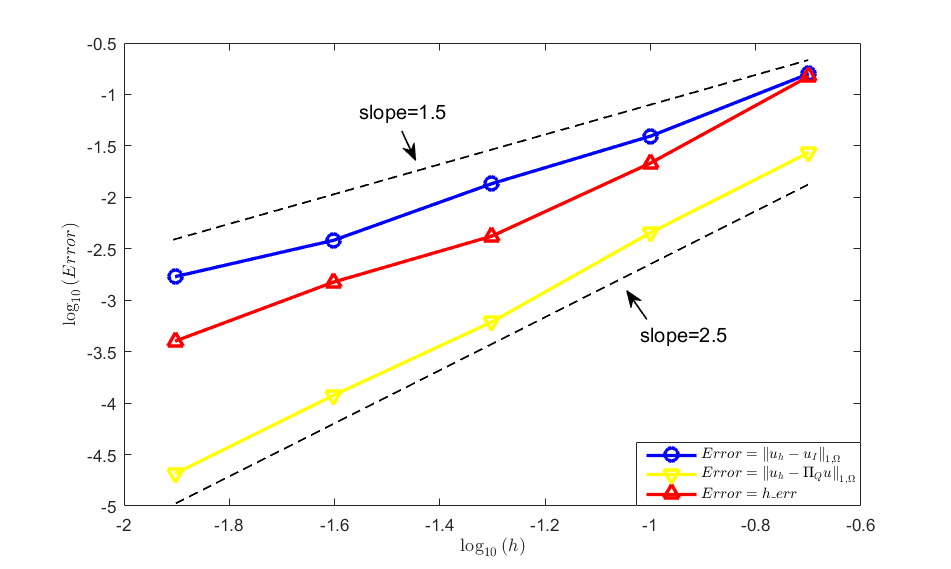}
	}
	\subfloat[A regular pentagon region]{
		\includegraphics[width=0.4\textwidth]{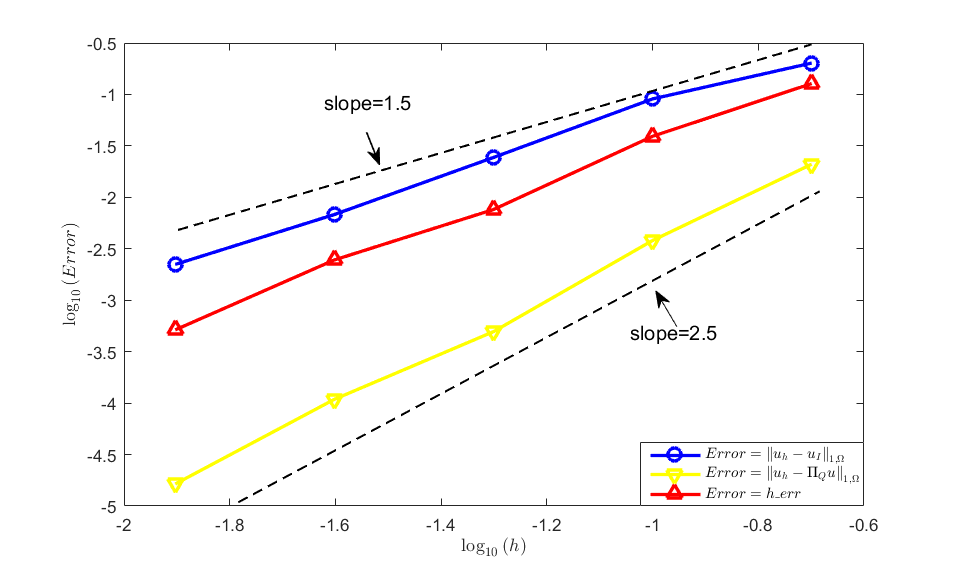}
	}
	\caption{${\left\| {\nabla \left( {{u_I} - {u_h}} \right)} \right\|_{0,\Omega }}$ and $h\_err$ versus the given size $h$ on BPM-based grids. }\label{error_cp}
\end{figure}

\subsection{Example 3.3: A regular pentagon region}
	\begin{figure}[h]
		\centering
		\subfloat[$h=0.2$]{
			\includegraphics[width=0.3\textwidth]{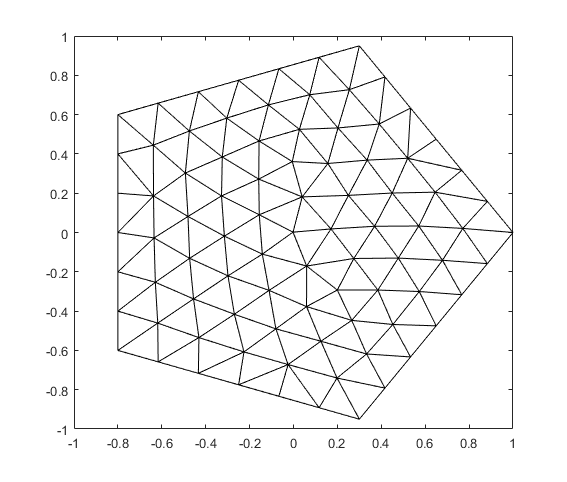}
		}
		\subfloat[$h=0.05$]{
			\includegraphics[width=0.3\textwidth]{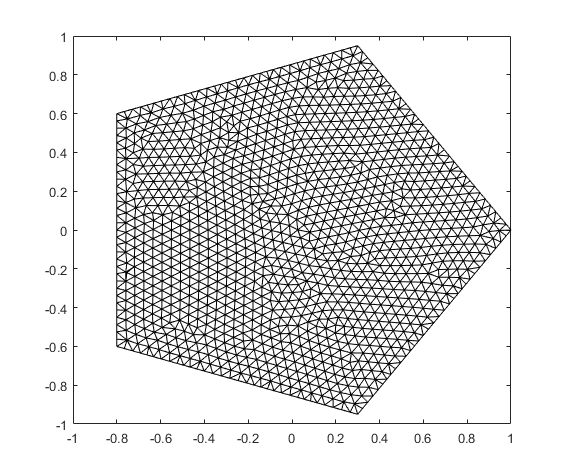}
		}
		
		\caption{ BPM-based grids on regular pentagon region in different sizes.}\label{grids_p}
	\end{figure}
	\begin{table}
		\caption{Results for regular pentagon region }
		\centering
		\small
		\begin{tabular}{cccccc}
			\toprule
			$h$  & ${\left\| {{u_h} - {u_I}} \right\|_{1,\Omega }}$ & order (k=1) & ${\left\|  {{u_h} - {\Pi_Q u}} \right\|_{1,\Omega }}$ & order (k=2) & ${Q_{avg}}$\\
			\midrule
			0.2   &  9.52E-02 &       & 9.37E-03 &     & 0.9582 \\
			0.1   &  3.34E-02 &  1.51 & 1.70E-03 & 2.46& 0.9651 \\
			0.05  &  1.09E-02 &  1.62 & 2.77E-04 & 2.62& 0.9608 \\
			0.025 &  3.74E-03 &  1.54 & 4.86E-05 & 2.51& 0.9670 \\
			0.0125 & 1.25E-03 &  1.58 & 8.31E-06 & 2.55& 0.9711 \\
			\bottomrule
		\end{tabular}\label{tab_p}
	\end{table}

Similarly, mesh size decreases by half in turn. For avoiding needless duplication, just choose two typical graphs displaying in Figure \ref{grids_p}. Choosing the exact solution $u=e^{x+y}$, all the same, Table \ref{tab_p} shows errors, convergence order and etc. As can be seen, there are still superconvergence phenomenon ${\left\|  {{u_h} - {u_I}} \right\|_{1,\Omega }}={\cal O }\left( {{h^{1.55}}} \right)$ and ${\left\| { {{u_h} - {\Pi_Q u}}} \right\|_{1,\Omega }}= {\cal O}({{h^{2.50}}})$ on regular pentagon region, verifying superconvergence estimation as well.

Overall, these numerical experiments fully show that there exists the superconvergence property on BPM-based grids. This is perhaps the results of most practical significance.

\section{Discussion}
\subsection{The superconvergence of problems with singularities}
By above experiments, we have observed superconvergence on BPM-based grids just for convex domain since superconvergence estimation requires $u \in {H^3}\left( \Omega  \right) \cap W_\infty ^2\left( \Omega  \right)$ which rules out domains with a re-entrant corner. In practice, it is well known that the solution may have singularities at corners. At this point, someone must have a question that will it appear superconvergence on problems with singularities? Although this paper has no supporting theory, we can find some illumination from the work by Wu and Zhang \cite{Wu2007Can}. They had deduced superconvergence estimation on domains with re-entrant corners for Poisson equation on mildly structured grids. The estimation is:
\begin{equation}
\begin{cases}
{\left\| {{u_h} - {u_I}}  \right\|_{1,\Omega }} \lesssim {N^{- \frac{1}{2}-\rho }},\\
{\left\| {{u_h} - {\Pi_Q u}}  \right\|_{1,\Omega }} \lesssim {N^{-1-\rho }},		
\end{cases}
\end{equation}
where $\rho$ is related to some mesh parameters. Additionally, for a 2D second-order elliptic equation, the optimal convergence rate is
\begin{equation}
\left\| {{u_h} - u} \right\|_{1,\Omega } \lesssim
\begin{cases}
N^{-\frac{1}{2}} \quad k=1,\\
N^{-1}  \quad k=2,\\	
\end{cases}
\end{equation}
where $k=1$ for the linear element and $k=2$ for the quadratic. $N$ is the total number of degrees of freedom to measure convergence rate in order to be used in adaptive finite element since the mesh is not quasi-uniform. Next, we will investigate the superconvergence on the L-shape region. 


Just choose a typical graphs displaying in Figure \ref{grids_L}(a). Specially, meshes near re-entrant corner are in good shape, illustrating in Figure \ref{grids_L}(b). The boundary conditions are chosen so that the true solution is ${r^{{2 \mathord{\left/{\vphantom {2 3}} \right. \kern-\nulldelimiterspace} 3}}}\sin {\textstyle{2 \over 3}}\left( {\theta  + {\textstyle{\pi  \over 2}}} \right)$ in polar coordinates. Figure \ref{error} demonstrates the relationship between superconvergence results and the total number of degrees of freedom. Notice that ${\left\| {{u_h} - {u_I}} \right\|_{1,\Omega }}$ and ${\left\| {{u_h} - {\Pi_Q}} \right\|_{1,\Omega }}$ are all superconvergent, which is consistent with estimation (4.1).
\begin{figure}[h]
	\centering
	\subfloat[L-shape region]{
		\includegraphics[width=0.3\textwidth]{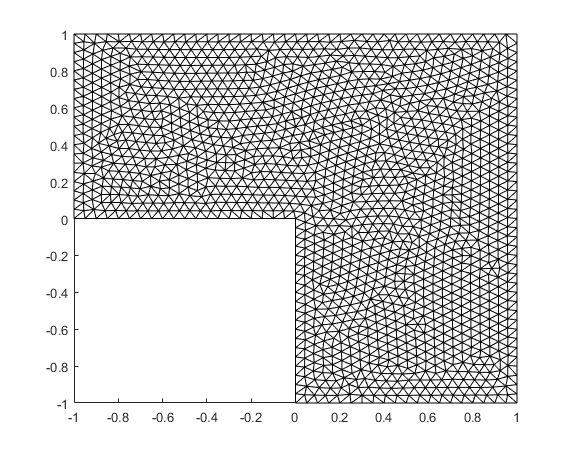}
	}
	\subfloat[local amplification of re-entrant corner]{
		\includegraphics[width=0.3\textwidth]{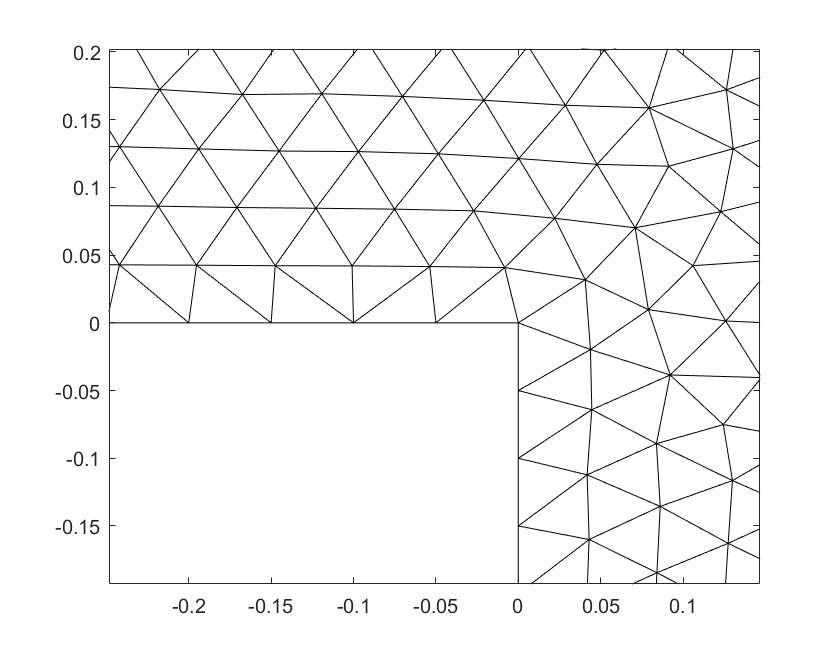}
	}
	\caption{BPM-based grids on L-shape region with $h=0.05$.}\label{grids_L}
\end{figure}

\begin{figure}
	\centering
	\subfloat{
		\includegraphics[width=0.5\textwidth]{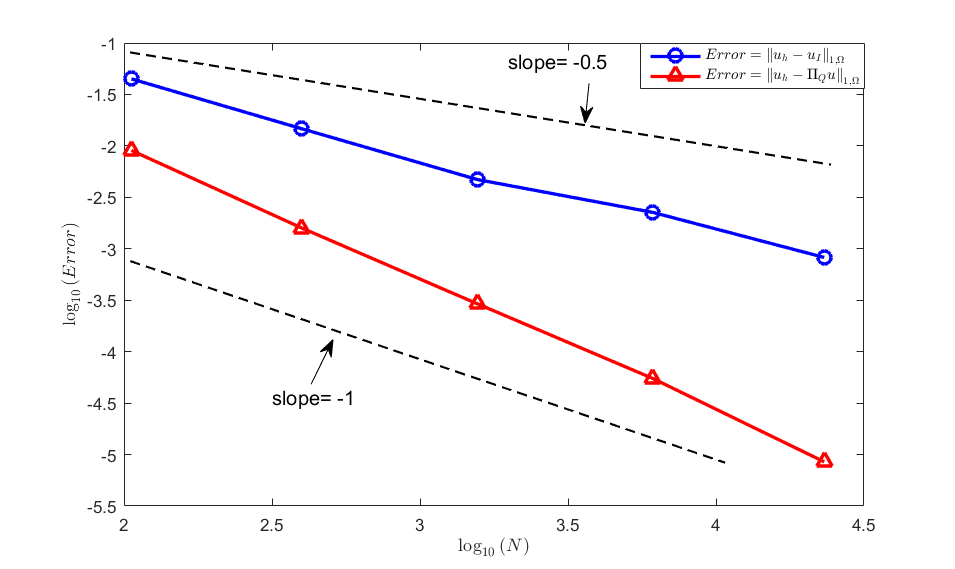}
	}
	\caption{${\left\| {{u_h} - {u_I}} \right\|_{1,\Omega }}$ and ${\left\| {{u_h} - {\Pi_Q u}} \right\|_{1,\Omega }}$ versus the total number of degrees of freedom on L-shaped domain. Dotted lines give reference slopes.}\label{error}
\end{figure}	

\subsection{The modification of the mesh condition}
\begin{table}
	\caption{ Results for four types of regions}
	\centering
	\small
	\begin{tabular}{ccccc}
		\toprule
		Domain & $mean$ & $var$ & $\mathop {\max }\limits_{e \in {\cal E}} \left| {{l_e} - h} \right|$ & $h_{err} = {\textstyle{{\sum {\left| {{l_e} - h} \right|} } \over {\# {\cal E}}}}$ \\
		\midrule
		Unit equilateral triangle & 0.1000            & 7.0481e-14         & 4.6798e-07           &1.6826e-07  \\
		Unit circle  & 0.0965            & 5.0513e-5         & 0.0184           &0.0093  \\
		Regular pentagon  & 0.0955            & 7.3520e-5         & 0.0213           &0.0087  \\
		L-shape  & 0.0946            & 2.0061e-4         & 0.0301           &0.0120  \\
		\bottomrule
	\end{tabular}\label{Tab1}
	\begin{tablenotes}
		\item[1]	
		Note: $Mean$ denotes the mean value of all edges' actual lengths, and $var$ is their variance value. Let $\cal E$ be the set of edges, ${l_e}$ the length of edge $e$, so $h_{err}$ represents the mean value of all edges' errors.
	\end{tablenotes}
\end{table}

In fact, not all edges of BPM-based grids could satisfy the mesh condition $|l_e-h|={\cal O}(h^{1+\alpha})(\alpha>0)$. There will be a few edges with poor performance in the actual computations especially for 'non-ideal subdivision' case. Nodes are placed in previous four types of computing regions with $h=0.1$, some useful results of these calculations are presented in Table \ref{Tab1}.

As shown, these mean values of the actual length are very closed to 0.1. However, compared with the mean error $h_{err}$, the maximal error $\mathop {\max }\limits_{e \in {\cal E}} \left| {{l_e} - h} \right|$ is underperformed except the equilateral triangle region, implying that there are 'bad' edges with larger errors in actual computations. In fact, the bubble system just reaches an approximate force-equilibrium state instead of the desired one we expect because of some numerical errors, so the bubble fusion degree of a few bubbles is slightly larger rather than a constant analyzed in Section 2.1, which are reflected in the red box in Figure \ref{bubble_distribution}. The fact leads to errors distributed under unfair conditions. In other words, additional errors are assigned to the several elements, so that there are some edges in bad behavior that the actual length and the given size $h$ differ by same order of the parameter $h$. 

\begin{figure}[h]
	\centering
	\subfloat[A unit equilateral triangle]{
		\includegraphics[width=0.3\textwidth]{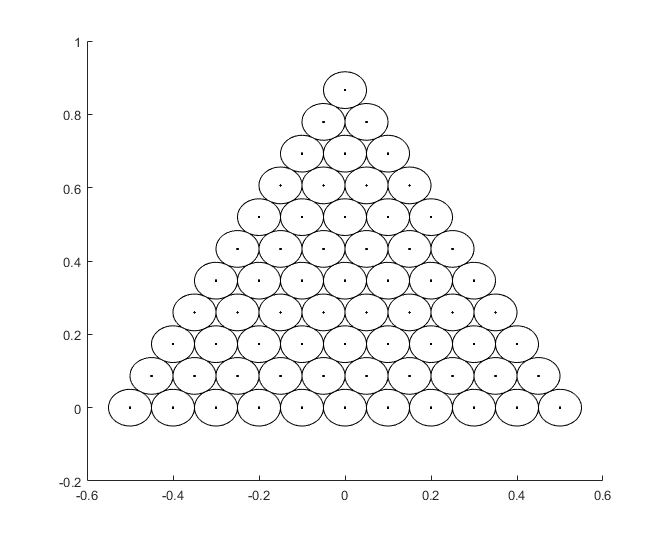}
	}
	\subfloat[A unit circle region]{
		\includegraphics[width=0.3\textwidth]{figures/bubble_3_6.png}
	}

	\subfloat[regular pentagon region]{
		\includegraphics[width=0.3\textwidth]{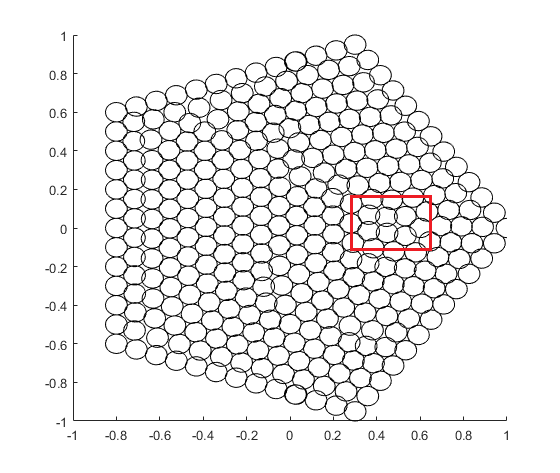}
	}
	\subfloat[L-shape region]{
		\includegraphics[width=0.3\textwidth]{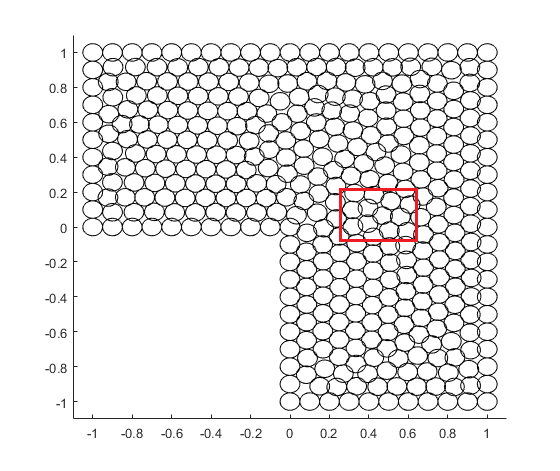}
	}
	\caption{ Bubbles distribution }\label{bubble_distribution}
\end{figure}
However, these bad edges are not too many since our algorithm guarantees the resultant external force of each bubble within the range of tolerance value closing to zero, and these are validated numerically by computational results that all variance values in Table \ref{Tab1} are very meager. In general, edges' lengths are basiclly in line with the size requirements, and there is no extreme situation that the actual length differs greatly from the given size. In the meantime, the maximal and average errors are all low values, which are also strong evidence.
%

Drawing from the above discussions and results, the mesh condition of BPM-based grids could be modified by the following changes. Let $e$ be an edge in the triangulation ${{\cal T}_h}$ derived from BPM-based grids and its length is represented by ${l_e}$. Denote ${\cal E}={{\cal E}_1} \oplus{{\cal E}_2}$ be the set of all edges belongs to triangulation ${{\cal T}_h}$ : 
\begin{enumerate}
	\item  For each ${e\in{\cal E}_1}$  
	\begin{equation}
	\left|{l_e} - h \right| = O\left( {{h^{1 + \alpha }}} \right),
	\end{equation}
	where $\alpha$ is a positive number.
	\item As for ${e\in{\cal E}_2}$ are in 'bad' group, $\left| {l_e} - h \right| =O\left( h\right) $, but for the two elements $\tau$ and $\tau '$ sharing edge $e$, they satisfy: 
	\begin{equation}
	\sum\limits_{e \in {{\cal E}_2}} {(\left| \tau  \right|+\left| \tau '  \right|) = O\left( {{h^{2\sigma }}} \right)},
	\end{equation}
	or
	\begin{equation}
	\sharp {{\cal E}_2}  \lesssim N^{\sigma},
	\end{equation}	
	where $\sigma$ is a positive number, $\tau$ and $\tau '$ are the two elements sharing edge $e$, $N$ is the total number of all edges.
\end{enumerate}
The second condition is infected by considerations of edges with poor performance. The values of maximal error show us that there are still some edges with larger errors, but the mean and variance of all edges' actual length indicate that most edges put up a good showing with rarely or no extreme situations. For two different expressions in (4.4) and (4.5), they are essentially equivalent. And they all reflect 'bad edges' are small percentage of all edges. Combining the numerical analysis, we add the 'bad edges' group to the previous theoretical results (2.6), such that the description of BPM-based grids is more realistic. 

The accession of 'bad edges' group to our mesh condition will have effect on theoretical superconvergence estimations, but it is negligible. Just the expressions of estimations have been slightly modified:
\begin{equation}
{\left\| {{u_h} - {u_I}} \right\|_{1,\Omega }} ={\cal O} (h^{1+\min (\alpha,\sigma, 1/2 )}),
\end{equation}
and
\begin{equation}
{\left\| {{u_h} - {\Pi_Q u}} \right\|_{1,\Omega }} ={\cal O}(h^{2+\min (\alpha, \sigma, 1/2 )}).
\end{equation}
Note that 'bad edges' group has been considered into mildly structured grids by Xu \cite{Xu2003Analysis}, so it will not create difficulties to theoretical derivations. In particular, the superconvergence property of numerical experiments in Section 3 don't get affected. Actually, 'bad edges' group is introduced just for describing BPM-based grids in line with the actual computation.

\section{Conclusion}
By analysing the properties of BPM, some mesh conditions of BPM-based grids on any bounded domain are derived. It is well to be reminded that our mesh conditions can be applied to the different superconvergence estimation analyzed by many scholars, like R. E. Bank, J. Xu, H. Wu and etc. As a result,  superconvergence estimation is discussed on BPM-based grids as the second work in this paper. Notably, this is the first time that the mesh conditions are theoretically derived and successfully applied to the existing estimations. Further, our conclusions can be applied to the posterior error estimation and adaptive finite element methods for improving precision of finite element solution. 

Though the initial study on the superconvergence phenomenon of BPM-based grids is made for a classical model of two-dimensional second order elliptic equation, we will consider concrete superconvergence post-processing as well as its theoretical estimation on BPM-based grids and expect greater advantages when solve more complex systems of equations.

\section*{Acknowledgments}
This research was supported by National Natural Science Foundation of China (No.11471262 and No.11501450) and the Fundamental Research Funds for the Central Universities (No.3102017zy038). Dr. Nan Qi would like to thank "the Fundamental Research Funds" of the Shandong University.

\section*{Reference}
\bibliographystyle{elsarticle-num}
\bibliography{Reference.bib}

\end{document}